\documentclass[11pt]{article}
\usepackage[letterpaper,margin=1in]{geometry}
\usepackage{amsmath,amssymb,amsthm,microtype}
\usepackage[T1]{fontenc}
\usepackage{lmodern}
\usepackage[colorlinks=true,linkcolor=blue,citecolor=blue,urlcolor=blue]{hyperref}
\newcommand{\remove}[1]{}
\newtheorem{theorem}{Theorem}
\newtheorem{lemma}{Lemma}
\newcommand{\E}{\mathbb E}
\newcommand{\Prob}{\mathbb P}
\newcommand{\Leb}{\mathrm{Leb}}

\DeclareMathOperator{\clip}{clip}
\DeclareMathOperator{\Var}{Var}

\title{Small Influential Coalitions in $[0,1]^n$ via the Junta Theorem}
\author{Ehud Friedgut}
\date{Revised 30 September 2026}
\begin{document}
\maketitle
\begin{abstract}
Every monotone Boolean function on the continuous cube admits a coalition of
$O(n/(\varepsilon\log n))$ coordinates that can force a fixed output (either zero or one) with
probability at least $1-\varepsilon$. This old conjecture of mine was recently proven by Chattopadhyay and Gurumukhani~\cite{CG}. This writeup contains a simplification of their proof, generated by AI after suggesting the use of the junta theorem of~\cite{F98} and the discretization procedure of ~\cite{F04}. 
\end{abstract}
\section{introduction}
My initiation to the world of Boolean functions began when I was a first year graduate student, with me stumbling on a published mistake, and having an idea about how to fix it. Several years later I published this conjecture\cite{F04}, and after more than two decades it was recently proven by Chattopadhyay and Gurumukhani~\cite{CG}. The purpose of this note is to offer a simplified version of their proof, discovered with the aid of ChatGPT-6 Astra, following some laconic suggestions that I made.

During my first year as a graduate student at the Hebrew University I was taking a seminar with Gil Kalai, and was supposed to present two papers to the class: KKL~\cite{KKL}, and BKKKL~\cite{BKKKL}.  

KKL states (theorem 1) that any Boolean function $f$ on $\{0,1\}^n$ has a variable with influence of order at least $(\frac{\log n}{n}) \Var(f)$, where $\Var(f)$ is the variance of the function. Repeated applications of this theorem yield theorem 2: If $f$ is a Boolean function with expectation, say, $1/2$, there's a coalition of order $\frac{n}{\log(1/\varepsilon) \log n }$ that for at least a $(1-\varepsilon)$ proportion of the inputs can force the function to be either zero or one, as it wishes. The implication from theorem 1 to theorem 2, for a monotone function, is simple. To influence towards 1 take the most influential variable, set it to be equal to one, resulting in an increase of order at least $\log n /n$ in the expectation, rinse, lather, and repeat, resulting in a coalition with strong influence towards 1. Then do the same to get the other half of the final coalition, a set of coordinates that can jointly influence towards 0.

BKKKL has the same two claims, theorem 1 and 2, for functions $f:[0,1]^n \rightarrow {0,1}$. The proof of theorem 1 is a bit more involved than in the discrete case, but for theorem 2 the authors simply offer the observation "by repeated applications of theorem 1, one has:...".  

When first reading BKKKL, after understanding the proofs of KKL theorems 1 and 2, and understanding BKKKL theorem 1, I took BKKKL's assurance for theorem 2 at face value, not giving it a second thought. However, on the evening before presenting the material to the class, while going over the material again, I was stumped; In the discrete  case, when setting a varaiable to equal 1, the effect is that on every fiber where that variable had influence, the expectation of $f$ increased from 1/2 to 1. However, in the continuous case, the initial conditional expectation on a fiber can already be very close to 1 (or to 0), so no apriori bound on the increase in expectation is guaranteed. A useful example to consider, is the one given in the paper BKKKL to illustrate the difference between the discrete and continuous cases: the indicator function of $[1/n,1]^n$. Here all coordinates have constant influence, but setting a coordinate to 1 results in an increase of order only $1/n$. After several hours of unsuccessful attempts at proving BKKKL theorem 2, I realized that not only is the statement incorrect, the example of the aforementioned function, given in the paper, is a counterexample! It is easy to verify that if a set of coordinates has size $o(n)$, than setting all of them to be equal to 1 only increases the expectation of $f$ by $o(1)$. (Gil once mused on this unfortunate glitch, and asked whether a student who proves a theorem in an exam, and also supplies a counterexample should be given partial credit for their effort.) 

After recovering from the shock of my discovery, and being encouraged by the 40\% of the authors of BKKKL who were at the Hebrew University (and subsequently became my mentors), I asked myself whether the situation can be partially salvaged. In the counterexample it's true that any small coalition is unable to force the function's expectation towards 1, but even a single coordinate set to 0 forces the function to be equal to 0. I asked, then, is it true that there's always a small coalition of coordinates that can force the function either to be 1 or to be 0 for most inputs, what's known as one-sided-influence. Several years later, in \cite{F04}, I put this conjecture in print, and it has since, for over two decades, been a constant companion of mine - I have returned to think about it many times throughout the years, including unsuccessful recent attempts using AI.

A couple of weeks ago, as part of the flurry of AI-generated results, I was delighted to see that Chattopadhyay and Gurumukhani of Cornell University had finally managed to put the conjecture to rest, in a 15 page manuscript ~\cite{CG}. One of the main tools that they used was the extremely deep result of Hatami ~\cite{Hatami} , known as Hatami's pseudo-junta theorem. After digesting their proof for several days, it occured to me that perhaps this was overkill, and that the proof could be simplified by using my junta theorem from ~\cite{F04}, a much weaker result than Hatami's, but perhaps sufficient for the task at hand. I therefore fed the C-G manuscript into ChatGPT-6 Astra and supplied the following prompt:

"{\it{
Please read the attached manuscript, and see whether it can be simplified by using Friedgut's junta theorem instead of Hatami's pseudo-junta theorem.}}"

After a few minutes the bot supplied me with a simpler proof, per its instructions. I shared this new simplified version with Chattopadhyay and Gurumukhani, and suggested that they update their version. After a few more days, I had another idea, and suggested to the bot that it read \cite{F04}, the paper where the original conjecture was suggested, and supplied it with the following prompt:
"{\it{
Please read ~\cite{F04} and see whether you can mimic the discretization there, to go directly from $[0,1]^n$ to $\{0,1\}^n$ without going through $\{1,\ldots, r\}^n$, as you did in the previous manuscript.}}"

This induced a further simplification, resulting in the proof included in this note. I communicated this further progress to Chattopadhyay and Gurumukhani, and they suggested that perhaps I should post it, and indeed, the content of the next section of this notevis the output that ChatGPT supplied me with.

I would like to stress that my mathematical contribution to this current project was limited to my two suggestions to the bot that I mentioned above. I did not apply any creativity to the project, just suggested to the bot which tools it might want to use. Furthermore, and this is important in giving Chattopadhyay and Gurumukhani the credit for the breakthrough, I had previously tried using ChatGPT (albeit a weaker model) to solve this problem, and was unsuccessful.

I lightly edited the following section, which was written by ChatGPT.

\section{The statement and proof}
Throughout, monotone means coordinatewise nondecreasing. For
$B\subseteq[n]$ and $b\in\{0,1\}$, write $b_B$ for the constant
assignment $b$ on $B$. A coalition $B$ uses this assignment to force
output $b$, while the remaining coordinates $X_{-B}$ are independent and
uniform on $[0,1]$. The coalition and its target are chosen before these
remaining coordinates are sampled. All logarithms are natural.
\subsection{The statement of the theorem}
\begin{theorem}\label{thm:main}
There is an absolute constant $C$ such that, for $n\ge2$ and
$0<\varepsilon<1/2$, every monotone
$f:[0,1]^n\to\{0,1\}$ admits $B\subseteq[n]$ and $b\in\{0,1\}$ with
\[
 |B|\le \frac{Cn}{\varepsilon\log n},
 \qquad \Prob[f(b_B,X_{-B})=b]\ge1-\varepsilon.
\]
\end{theorem}

\paragraph{Remark.}
The reduction to the monotone case by coordinatewise compression is
standard; see~\cite{F04,CG}.

The argument builds on the influential-coalition framework of
Chattopadhyay and Gurumukhani~\cite{CG} and uses the discretization
in~\cite[Claim~2.1]{F04} and the continuous-junta observation
in~\cite[Section~2.2]{F04}.

\subsection{Low total influence implies a junta with large one-sided influence}
{\bf{Ehud's commentary:}} What's shown here is that there's a one-sided analogy of the junta theorem in the continuous setting, i.e., if the total influence of the function is bounded by $K$, there's a coalition with size exponential in $K$ that has large one-sided influence. The proof is a minor variant on the modern explanation of how KKL implies BKKKL.

For monotone $h:[0,1]^n\to\{0,1\}$, put
\[
 K(h)=\sum_{i=1}^n
 \E\bigl[h(1,X_{-i})-h(0,X_{-i})\bigr].
\]
Thus $K(h)$ is the total probability of nonconstant coordinate fibres,
including fibres whose change occurs only at an endpoint.

\begin{lemma}\label{lem:junta}
There is an absolute constant $c$ such that, if $K(h)\le L$ with $L>0$,
then for every $0<\eta<1$ there are $S\subseteq[n]$ and $b\in\{0,1\}$ with
\[
 |S|\le\exp(cL/\eta),\qquad
 \Prob\bigl[h(b_S,X_{-S})=b\bigr]\ge1-\eta.
\]
\end{lemma}

\begin{proof}
For $M=2^m$, let $D_M(x)=\lfloor Mx\rfloor/M$ for $x<1$, with
$D_M(1)=(M-1)/M$. Set $h_M(x)=h(D_M(x_1),\ldots,D_M(x_n))$ and regard
this step function as a Boolean function $H_M$ on $nm$ fair bits.
Monotonicity gives
\begin{equation}\label{eq:approx}
 \Prob[h\ne h_M]\le n/M.
\end{equation}
Indeed, on each small cube both values lie between the values of $h$ at
its lower and upper corners. Summing these corner differences, and
telescoping one coordinate at a time, gives at most $nM^{n-1}$.

With all other blocks fixed, $H_M$ is a threshold in the binary order on
the remaining block. Its $j$th bit has influence at most $2^{-(j-1)}$:
only the unique prefix of length $j-1$ containing the threshold can be
pivotal. Thus a non-constant block section contributes less than $2$ to
the total bit influence $I(H_M)$, where each bit's influence is the
probability that flipping it changes the output. Consequently
\begin{align}
 I(H_M)
 &\le 2\sum_i\E\bigl[h(1,D_M(X_{-i}))-h(0,D_M(X_{-i}))\bigr]\notag\\
 &\le 2K(h)+\frac{2n(n-1)}{M}.\label{eq:bits}
\end{align}
For the last inequality, apply~\eqref{eq:approx} in dimension $n-1$ to
each endpoint section; the lower-grid average of the upper section is
at most its integral, and that of the lower section is at least its
integral minus $(n-1)/M$.

Take $M$ sufficiently large that the error in~\eqref{eq:approx} is at most
$\eta/2$ and $I(H_M)\le4L$. Friedgut's Boolean junta theorem~\cite{F98}
gives an $\eta/2$-approximation to $H_M$ depending on at most
$\exp(8c_0L/\eta)$ bits, where $c_0$ is its absolute constant.
Let $S$ be the original coordinates containing those bits. The resulting
function $Q(X_S)$ approximates $h$ with error at most $\eta$.
Averaging over $X_S$ gives a fixed assignment $a$ for which
\[
 \Prob[h(a,X_{-S})\ne Q(a)]\le\eta.
\]
Set $b=Q(a)$. Moving every coordinate of $a$ to the endpoint $b$ can only
increase the probability of output $b$. This proves the lemma with
$c=8c_0$.
\end{proof}

\subsection{Proof of Theorem 1}
{\bf{Ehud's commentary:}}
When trying to prove the main theorem, what I believe that any mathematician would attempt at first is the following idea; Find a variable with large influence towards zero. Set it to be equal to zero. Repeat until all variables have small influence towards zero. And then... and then what? What you would hope for is to prove that then there's a small coalition with decisive influence towards 1, but that's the source of difficulty of this conjecture - that implication seems to be as hard as the original problem. The ingenious fix that Chattopadhyay and Gurumukhani found involved using the heavy tool of Hatami's pseudo junta. What ChatGPT eventually came up with, as described below, is really a beautiful idea. It describes a continuum of measures that are perturbations of the uniform measure, and proves that for at least one of them the function we're studying has low total influence, and hence by the lemma proven above has a small junta with large one sided influence. It then elegantly manages to move back to the uniform measure, while slightly modifying the junta at random, in order to simulate the perturbed measure in the world of the uniform measure.
This trick reminds me a bit of the Dinur-Safra idea that they coined as "The Importance of Being Biased" in ~\cite{DS}  where they get great mileage out of the following observation: every monotone function $f$ is close to a junta with respect to the product measure $\mu_p^{\otimes n}$ for some value of $p$ (because as long as it's not approximable by a junta, the expectation of $f$, as a function of $p$ has large derivative, but it can't continue climbing for too long, since it's bounded by 1).

\paragraph{}
{\bf{Proof}:}
Let $f$ be monotone, and let $\mu$ denote the uniform measure on $[0,1]$. Fix $0<\rho<1$ and, for $0\le t\le\rho$, set
\begin{equation}\label{eq:measure}
 \nu_t=(1-\rho)\mu+(\rho-t)\delta_0+t\delta_1,
 \qquad F(t)=\E_{\nu_t^{\otimes n}}f.
\end{equation}
The function $F$ is a polynomial in $t$. Differentiating the product measure,
\begin{equation}\label{eq:derivative}
 F'(t)=\sum_i\E_{\nu_t^{\otimes(n-1)}}
                 [f(1,X_{-i})-f(0,X_{-i})]=:K_t.
\end{equation}
Since $\int_0^\rho K_t\,dt=F(\rho)-F(0)\le1$, there is
$t\in(0,\rho)$ with $K_t\le1/\rho$. Fix such a $t$.

To use Lemma~\ref{lem:junta} under $\nu_t$, define
\[
 T_t(u)=\clip_{[0,1]}\left(\frac{u-\rho+t}{1-\rho}\right),
 \qquad h(u)=f(T_t(u_1),\ldots,T_t(u_n)).
\]
Here $\clip_{[0,1]}(v)=\max\{0,\min\{v,1\}\}$.
The map $T_t$ sends Lebesgue measure to $\nu_t$, preserves the endpoints,
and is increasing. Hence $h$ is monotone and $K(h)=K_t\le1/\rho$.
Applying the lemma with $L=1/\rho$ and $\eta=\varepsilon/2$ gives $S,b$ with
\begin{equation}\label{eq:small-fibre}
 |S|\le\exp\left(\frac{A}{\rho\varepsilon}\right),\qquad
 \Prob_{\nu_t^{\otimes([n]\setminus S)}}
           [f(b_S,X_{-S})\ne b]\le\varepsilon/2,
\end{equation}
for an absolute $A\ge1$.

Now replace both endpoint atoms in~\eqref{eq:measure} by the favourable
endpoint $b$. Monotonicity shows that~\eqref{eq:small-fibre} still holds
under
\[
 \nu^{(b)}=(1-\rho)\Leb+\rho\delta_b.
\]
Sample this measure by putting each coordinate outside $S$ into a random
set $R$ independently with probability $\rho$, fixing those coordinates
to $b$, and sampling all the others uniformly. If
\[
 e(R)=\Prob\bigl[f(b_{S\cup R},X_{-(S\cup R)})\ne b\bigr],
\]
then $\E_R e(R)\le\varepsilon/2$. Also
$\Prob[|R|\le2\rho(n-|S|)]\ge1/2$ by Markov's inequality
(with the case $S=[n]$ immediate). On this event the conditional average
of $e(R)$ is at most $\varepsilon$. Thus some fixed $R$ satisfies both
bounds. The set $B=S\cup R$ therefore has
\begin{equation}\label{eq:core}
 |B|\le\exp\left(\frac{A}{\rho\varepsilon}\right)+2\rho n,
 \qquad \Prob[f(b_B,X_{-B})=b]\ge1-\varepsilon.
\end{equation}
All choices, including $b$ and $R$, precede sampling the remaining
uniform coordinates.

If $\varepsilon\log n>2A$, choose $\rho=2A/(\varepsilon\log n)$.
Then~\eqref{eq:core} gives
\[
 |B|\le\sqrt n+\frac{4An}{\varepsilon\log n}
       \le\frac{(4A+1)n}{\varepsilon\log n}.
\]
If $\varepsilon\log n\le2A$, using all $n$ coordinates suffices.
This proves Theorem~\ref{thm:main}.

\end{document}